\documentclass[10pt,a4paper,reqno]{amsart}

\usepackage{amsmath,amssymb,amsthm,amsfonts}
\usepackage{amsbsy}
\usepackage{mathtools}

\usepackage[utf8]{inputenc}
\usepackage[symbol]{footmisc}

\usepackage[norelsize, boxruled]{algorithm2e}
\usepackage{mathtools}

\usepackage[english]{babel}
\usepackage{url}
\usepackage{wrapfig}
\usepackage{multicol}
\usepackage{enumerate}
\usepackage{subcaption}

\usepackage{tikz}

\usepackage[a4paper,left=25mm,right=25mm,top=30mm,bottom=30mm]{geometry}

\usepackage{xcolor}

\definecolor{LinkColor}{rgb}{0,0,0} 
\usepackage[colorlinks=true,linkcolor=LinkColor,citecolor=LinkColor,urlcolor= LinkColor, naturalnames, hyperindex, pdfstartview=FitH, bookmarksnumbered, plainpages]{hyperref}

\newtheorem{theorem}{Theorem}[section]

\newtheorem{lemma}[theorem]{Lemma}
\newtheorem{proposition}[theorem]{Proposition}
\newtheorem{question}[theorem]{Question}

\newtheorem*{question*}{Question}

\theoremstyle{definition}

\newtheorem{remark}[theorem]{Remark}

\usepackage[normalem]{ulem}

\title{The Modular Isomorphism Problem for small groups - revisiting Eick's algorithm}
\author[L.~Margolis]{Leo Margolis}
 \address{Vakgroep Wiskunde, Vrije Universiteit Brussel, Pleinlaan 2, 1050 Brussels, Belgium}
 \email{\href{mailto:leo.margolis@vub.be}{leo.margolis@vub.be}}

\author[T.~Moede]{Tobias Moede}
 \address{Institut f\"ur Analysis und Algebra, Technische Universit\"at Braunschweig, Universit\"atsplatz 2, 38106 Braunschweig, Germany}
 \email{\href{mailto:t.moede@tu.bs}{t.moede@tu.bs}}

\keywords{Modular Isomorphism Problem, modular group algebras, $p$-groups, small groups}
\subjclass[2010]{20C05, 20D15, 20C40, 16L60}
\thanks{This research has been supported by the Research Foundation Flanders (FWO - Vlaanderen).}

\begin{document}
\maketitle

\begin{abstract}
We study the Modular Isomorphism Problem (MIP) for groups of small order based on an improvement of an algorithm described by B. Eick. 
Our improvement allows to determine quotients $I(kG)/I(kG)^m$ of the augmentation ideal without first computing the full augmentation ideal $I(kG)$. 
It allows us to verify that the MIP has a positive answer for groups of order $3^7$ and to significantly reduce the cases
that need to be checked for groups of order $5^6$. We further provide a proof for an observation of Bagi\'nski and provide a negative answer to a
question of Bleher, Kimmerle, Roggenkamp and Wursthorn.
\end{abstract}

\section{Introduction}

\subsection{The Modular Isomorphism Problem}~\par
A classical algebraic problem asks to determine when two non-isomorphic groups $G$ and $H$ might have isomorphic group rings over a given ring $R$. Most classical variations of this question 
have been answered, e.g. when $R$ is the ring of integers \cite{Her01}. We denote the group ring of $G$ over $R$ by $RG$. The most famous formulation which remains open today and goes back at least 
to Brauer \cite[\S 9 of Supplements]{Bra63} is the \\

\textbf{Modular Isomorphism Problem (MIP):} Let $p$ be a prime, $k$ a field of characteristic $p$ and let $G$ and $H$ be finite $p$-groups. Does $kG \cong kH$ imply $G \cong H$?\\

From now on we will assume that $G$ and $H$ are finite $p$-groups and that $k$ is a field of characteristic $p$.
Though it found considerable attention not many strong results were achieved on the MIP, at least compared to results on other isomorphism questions on group rings. We refer to \cite{HS06, EK11} for 
an overview of (most) known results and to \cite{BK19, Sak20, BdR20} for some very recent work on the problem. A block-theoretical version of MIP was recently considered in \cite{NS18}. The strongest 
generic result available for groups of small order is that the MIP has a positive answer for groups of order at most $p^5$ \cite{SS96}.

The MIP takes its strongest form, when one assumes that $k$ is the field with $p$ elements, in which case $kG$ is a finite object. Hence an algorithmic approach to the MIP for two given groups appears 
natural and two algorithms have been developed to deal with the problem: A cohomological algorithm by Roggenkamp and Scott was improved by Wursthorn and used to solve the MIP for 
groups of order $2^6$ and $2^7$ \cite{Wur90, Wur93, BKRW99}, but his implementation appears to be lost. A second algorithm was developed by Eick, implemented as a \texttt{GAP}-package \cite{GAP, ModIsom} 
and used to study groups of order $3^6$, $2^8$ and $2^9$ \cite{Eic08, EK11}. Eick's algorithm provides a canonical form for a nilpotent algebra $I$  starting with a structure constant table for $I$ and 
inductively computing canonical forms for the quotients $I/I^n$.

Recall that the augmentation ideal $I(kG)$ of $kG$ consists of the elements of $kG$  with coefficient sum equal to $0$. It is easy to see that $kG \cong kH$ if and only if $I(kG) \cong I(kH)$, 
cf. \cite[2. Corollary]{Eic08}. Now $I(kG)$ is a nilpotent ideal and thus Eick's algorithm can be used to find a canonical form for $I(kG)/I(kG)^n$ for any positive integer $n$. 
Our main improvement to Eick's algorithm is described in Section \ref{sec:Algorithm}. It facilitates the computation of a structure constants table for $I(kG)/I(kG)^m$  for some initial value $m$ 
(which can be increased if necessary), whereas Eick's version starts by calculcating a structure constants table for the full augmentation ideal $I(kG)$. For groups of order $3^7$ and $5^6$ 
this improvement turns out to be crucial as otherwise the calculations are too time and memory consuming. 

\subsection{Results}~\par
We summarize the results obtained using our improvement of Eick's algorithm, which is available as a \texttt{GAP}-package \texttt{ModIsomExt} \cite{ModIsomExt}. 

\begin{theorem}
MIP has a positive answer for groups of order $3^7$.
\end{theorem}

For the groups of order $5^6$ we reduce the verification of the Modular Isomorphism Problem to six cases. Here \texttt{SG($5^6$, i)} refers to the $i$-th group of order $5^6$
as indexed in the Small Groups Library of \texttt{GAP}.

\begin{theorem}
If there are non-isomorphic groups $G=SG(5^6,i)$ and $H=SG(5^6, j)$ of order $5^6$ such that $kG \cong kH$, then $\{i,j\}$ is a subset of one of the following sets: 
\[ \{183, 184, 185,186 \}, \{202,203,204,205\}, \{553,554 \}, \{580,594,595,596 \}, \{584,599,600,601 \}, \{628,629 \}. \]
\end{theorem}

While improving Eick's implementation we also noticed a small mistake in the program used to verify the MIP for groups of order $3^6$, $2^8$ and $2^9$ which leads to the necessity to re-examine these orders. 
Our computations yield the following result for groups of order $3^6$ and $2^8$.

\begin{theorem}
MIP has a positive answer for the groups of order $3^6$ and $2^8$.
\end{theorem} 

Apart from the improved efficiency, our version also allows to compare groups of any order. In \cite{ModIsom} it was only possible to apply the algorithms to 
groups which are contained in the Small Groups Library of \texttt{GAP}. This further extends the possibilities for computational investigations of the MIP and allows to develop ideas for theoretical proofs of the MIP based on experimental observations for bigger groups with special properties.

\subsection{Group-theoretical invariants} ~\par
An object is called a \textit{group-theoretical invariant} of $kG$, if it depends only on $G$ and is determined by the group algebra $kG$. Any program to verify MIP usually starts by comparing 
group-theoretical invariants and only compares the isomorphism types of the group algebras of those groups for which all invariants coincide. We include many invariants from the literature 
which were not included before; for an overview see Theorem~\ref{th:Invariants}. Notably, the following observation was made (without proof) by Bagi\'nski \cite{Bag99} and allows to eliminate 
some computationally difficult cases. We denote by $\gamma_i(G)$ the $i$-th term of the lower central series of $G$.

\begin{proposition}\label{prop:BaginskiInv}
Let $G$ be a finite $p$-group generated by two elements and assume $kG \cong kH$. Then $G/\gamma_2(G)^p\gamma_4(G) \cong H/\gamma_2(H)^p\gamma_4(H)$.
\end{proposition}

We give more details and a proof in Section \ref{sec:Bag}.

\subsection{The Jennings bound and a question by Bleher, Kimmerle, Roggenkamp \& Wursthorn} \label{sec:ques} ~\par
We denote by $D_i(G)$ the Jennings-Zassenhaus series of $G$ and for a pair of group $G$ and $H$ we define the \textit{Jennings bound} of $G$ and $H$ to be the maximal integer $s$ such that $G/D_s(G) \cong H/D_s(H)$. 
Let $s$ be the Jennings bound for $G$ and $H$. Then it follows from Jennings' construction of a basis of the augmentation ideal, cf. Section~\ref{sec:JenningsBound}, that  $I(kG)/I(kG)^s \cong I(kH)/I(kH)^s$. In particular when running Eick's 
algorithm for $G$ and $H$, one knows a priori that at least $I(kG)/I(kG)^{s+1}$ has to be considered.

Let $r$ be a minimal integer such that $I(kG)/I(kG)^r \not \cong I(kH)/I(kH)^r$. As necessarily $r > s$, if we want to say that $G$ and $H$ are ``close to being a counterexample to the MIP'', it appears 
appropriate to consider the relation between $s$ and $r$. In Section \ref{sec:JenningsBound} we give a negative answer to the following question by Bleher, Kimmerle, Roggenkamp and Wursthorn \cite{BKRW99}:

\begin{question*}\label{que:BKRW}
Let $G$ and $H$ be non-isomorphic $p$-groups [of the same order] such that $D_{n+1}(G) = 1$. Does $I(kG)/I(kG)^{2n+1} \not \cong I(kH)/I(kH)^{2n+1}$ always hold? 
\end{question*}

There are four pairs of groups of order $2^8$ that violate this bound; see Proposition \ref{prop:AnswerBKRW}. In the same section, we also pose a question involving the Jennings bound for odd primes.\\

\textbf{Notation:} As already stated, we denote by $\gamma_i(G)$ the lower central series and by $D_i(G)$ the Jennings-Zassenhaus series of $G$ (see Section~\ref{sec:Algorithm} for the definition of this series). 
The commutator subgroup $\gamma_2(G)$ is also denoted $G'$. By $Z(G)$ we denote the center of $G$ and by $\Phi(G)$ the Frattini-subgroup of $G$. For $g \in G$ we denote by $C_G(g)$ the centralizer of $g$ in $G$. For $g,h \in G$ 
we let $[g,h]$ be the commutator $g^{-1}h^{-1}gh$. We also write $g^h = h^{-1}gh$.

The \emph{augmentation homomorphism} of $kG$ is a ring homomorphism from $kG$ to $k$ which maps an element to the sum of its coefficients. The \textit{augmentation ideal} of $kG$, denoted by $I(kG)$, consists 
of all the elements of $kG$ of augmentation $0$. Moreover, we denote by $V(kG)$ the elements of $kG$ with coefficient sum 1. Note that any element of $V(kG)$ is a unit in $kG$ as $kG$ is a local algebra, 
so $V(kG)$ is the so-called \emph{normalized unit group} of $kG$.

\section{Group-theoretical invariants and known results}
We call a property of $G$ an invariant of $kG$ if any group $H$ such that $kG \cong kH$ also has this property (equivalently every group base $H \leq V(kG)$ has this property). Being the dimension of 
$kG$ as $k$-vector space the order of $G$ is clearly an invariant.

We give a list of many known invariants. We start by those contained in \cite{BKRW99} and then add further ones. Note that the invariants used by the \texttt{GAP}-package \texttt{ModIsom} are those listed in 
\cite[Theorem 7]{BKRW99}, these are exactly the points \eqref{item1} to \eqref{item8} in the following theorem. We do not always include references to the original proofs.
\begin{theorem}\label{th:Invariants}
The following properties of $G$ are invariants of $kG$. 
\begin{enumerate}[(a)]
\item\label{item1} The isomorphism type of $G/\Phi(G)$.
\item The isomorphism type of $G/G'$.
\item The isomorphism type of $Z(G)$.
\item The isomorphism type of $G/(\gamma_2(G)^p\gamma_3(G))$ (``Sandling quotient'').
\item The isomorphism types of $D_i(G)/D_{i+1}(G)$, $D_i(G)/D_{i+2}(G)$ and $D_i(G)/D_{2i+1}(G)$ for all $i \geq 1$ (Passi-Sehgal, Ritter-Sehgal).
\item The number of conjugacy classes of $G$ and the number of conjugacy classes of $p^n$-th powers for all $n\geq 0$ (K\"ulshammer).
\item $\sum_{g^G} \log_p(|C_G(g)/\Phi(C_G(g))| )$ where the sum runs over conjugacy classes of $G$ (``Roggenkamp parameter'').
\item\label{item8} The number of conjugacy classes of maximal elementary abelian subgroups (of any order) in $G$ (``Quillen invariant'').
\item The isomorphism type of $G/D_4(G)$, if $p \neq 2$ \cite{Her07}.
\item The number of conjugacy classes of $p^n$-th powers which have the same order as a class which powers to them (Parmenter-Polcino Milies) \cite[Corollary 2.4]{HS06}.
\item If $G$ is (elementary abelian)-by-abelian and $\gamma_{2p}(G) = 1$, then the isomorphism type of $\Phi(G)$ \cite[p. 16]{HS06}.
\item If $G'$ is cyclic, the size of the biggest cyclic subgroup containing $G'$ \cite[Proposition 4]{San96}.
\item If $G/C_G(G'/\Phi(G'))$ is cyclic, then the isomorphism types of $C_G(G'/\Phi(G'))/\Phi(G')$ \cite[Corollary 7]{Bag99} and $G/\Phi(C_G(G'/\Phi(G')))$ \cite[Theorem 9]{Bag99}.
\item The isomorphism types of $G' \cap Z(G)$ and $Z(G)/(G' \cap Z(G))$ \cite[Theorem 6.11]{San84}.
\item If $G$ is of class $2$, then the isomorphism type of $G/Z(G)$ \cite[Theorem 6.23]{San84}.
\item The isomorphism types of $D_i(G')/D_{i+1}(G')$ for all $i \geq 1$ \cite[Lemma 6.26]{San84}.
\item The nilpotency class of $G/\Phi(G')$ \cite[Proposition 1.2]{BC88}.
\item The nilpotency class of $G$ is an invariant if one of the following holds: $\text{exp}(G) = p$, the nilpotency class of $G$ equals $2$, $G'$ is cyclic \cite[Theorem 2]{BK07} or if $G$ is of maximal 
class \cite[Corollary 3.1]{BK19}.
\item If $G$ is $2$-generated, then the isomorphism type of $G/\gamma_2(G)^p\gamma_4(G)$. These are the last lines of \cite{Bag99} given without a proof. We include a proof in Section~\ref{sec:Bag}.
\end{enumerate}
\end{theorem}

We included all the invariants listed in Theorem~\ref{th:Invariants} in our program.
Of course also any cohomology group $H^n(G,A)$ is an invariant, for any $kG$-module $A$, but expressing this in terms of $G$ is usually not easy. If we denote by $k$ the trivial $kG$-module, then
the dimension of $H^1(G,k)$ is already included in Theorem~\ref{th:Invariants} as $G/\Phi(G)$. We also included the dimension of $H^2(G,k)$ in our program. 

For many classes of groups for which the MIP is known to hold, this follows from Theorem~\ref{th:Invariants}. This might be a direct application, e.g. if $D_3(G) = 1$, or combining several of the 
criteria listed, e.g. for metacyclic groups \cite{San96}. We list classes of groups for which MIP has been solved, but for which, to our knowledge, the results do not follow from Theorem~\ref{th:Invariants}.

\begin{theorem}
If $G$ is one of the following, then MIP has a positive answer for $G$.
\begin{itemize}
\item[(a)] $2$-groups of maximal class \cite{Bag92}.
\item[(b)] $p$-groups of maximal class and order at most $p^{p+1}$ which have a maximal abelian subgroup \cite{BC88}.
\item[(c)] $p$-groups with center of index $p^2$ \cite{Dre89}.
\end{itemize}
\end{theorem}
Moreover, MIP holds for certain $3$-groups of maximal class, but the conditions are too technical to be included here \cite{BK19}.

\section{The small group ring and Bagi\'nski's invariant}\label{sec:Bag}

A group-theoretical invariant which is very useful to us to eliminate some candidates to counterexamples which would have a very high Jennings bound 
is given without proof by Bagi\'nski in the last lines of \cite{Bag99}. He mentions that it should follow similar 
to Sandling's proof in \cite{San89}, but it is not evident how to achieve this. We give a proof of Bagi\'nski's observation using the algebra constructed by Sandling, but in the way it is described in \cite[Section 2.3]{HS06}. Bagi\'nski wrote us however that 
his proof is different from ours.

We first collect some standard results which will be useful to us.
\begin{lemma}\label{lemma_commids} Let $G$ be a group, $a, b, c \in G$ and $n$ a positive integer.
\begin{enumerate}
\item[(i)] The following identities hold \cite[Chapter III, 1.2 Hilfssatz]{Hup67}:
\begin{itemize}
\item $[a, b]^{-1} = [b, a]$.
\item $[a, bc] = [a,c][a,b]^c$.
\item $[ab,c] = [a,c]^b[b,c]$.
\end{itemize}
\item[(ii)] If $[a,b]$ commutes with $a$ and $b$ then $(ab)^n = a^n b^n [b,a]^{\binom{n}{2}}$ \cite[Chapter III, 1.3	 Hilfssatz]{Hup67}.
\end{enumerate}
\end{lemma}

\begin{proposition}\label{prop_smallgroupring} \cite{San89}
Let $k$ be the field with $p$ elements. Denote by $I$ the two sided ideal in $kG$ generated by elements of the form $(g-1)(n-1)$ where $g \in G$ and $n \in G'$. Let $g_1, \ldots, g_m$ be elements in $G$ whose 
images in $G/G'$ form an independent generating set of the abelian group $G/G'$. Let $S$ be the group of normalized units in $kG/I$. By abuse of notation denote the image of an element $u \in kG$ in $S$ also 
as $u$. Then:
\begin{itemize}
\item[(i)] The ideal $I$ does not depend on the choice of a group basis in $kG$. The isomorphism type of $kG/I$ does only depend on the isomorphism type of $G/\Phi(G')$. 
\item[(ii)] Assume from now on that $\Phi(G') = 1$. There is a natural short exact sequence
$$1 \rightarrow G' \rightarrow S \rightarrow \mathrm{V}(kG/G') \rightarrow 1. $$
Moreover $G$ is a normal subgroup in $S$.
\item[(iii)] Let $A$ be the subgroup of $S$ generated by elements of the form $1 + (g_1-1)^{k_1}(g_2-1)^{k_2}\ldots (g_m-1)^{k_m}$ where $k_1, \ldots, k_n$ are non-negative integers such that 
$k_1 + \ldots + k_m \geq 2$. Then $S = GA$ and the action of $a = 1 + (g_1-1)^{k_1}(g_2-1)^{k_2} \ldots (g_m-1)^{k_m} \in A$ on $G$ is given by
$$[g,a] = [g,g_1,g_1, \ldots ,g_1,g_2, \ldots ,g_2,g_3, \ldots g_m] $$ where in the above expression each $g_i$ appears exactly $k_i$ times.
\item[(iv)] Assume additionally that $\gamma_4(G) = 1$. Then $A$ is abelian and $S = G \rtimes A$. Moreover an independent generating set of $A$ is given by the elements 
$1 + (g_1-1)^{k_1}(g_2-1)^{k_2}\ldots (g_m-1)^{k_m}$ where $k_1 + \ldots + k_m \geq 2$, not all the $k_1,\ldots, k_m$ are divisible by $p$ and $k_i$ is smaller than the order 
of $g_i$ for each $1 \leq i \leq m$. The order of $1 + (g_1-1)^{k_1}(g_2-1)^{k_2}\ldots (g_m-1)^{k_m}$ is the smallest number $t$ such that $(g_i-1)^{k_it} = 0$ for some $1 \leq i \leq m$ for which $k_i \neq 0$.
\end{itemize}
\end{proposition}

The ring  $kG/I$ in the previous proposition is also known as the \emph{small group ring}.

\begin{proof}
These facts follow from the results in \cite{San89} and are explicitly given, though not in this compressed form, in \cite[Section 2.3, pp.15-16]{HS06}.
\end{proof}


We are now in the position to give a proof of Bagi\'nski's observation.

\begin{proof}[Proof of Proposition \ref{prop:BaginskiInv}]
Let $I$ be the two-sided ideal of $kG$ generated by the elements of the form 
$(g-1)(n-1)$ for $g \in G$ and $n \in G'$. Let $S$ be the normalized unit group of $kG/I$. As from now on all our calculations will take place in $kG/I$ we will denote the image of an element $u$ of $kG$ 
in $kG/I$ simply as $u$.


The isomorphism type of $kG/I$ depends only on the isomorphism type of $G/\Phi(G')$, so we can assume that $\Phi(G')= 1$, i.e. $G'$ is elementary abelian. From the facts given in 
Proposition~\ref{prop_smallgroupring} we easily derive that $\gamma_i(G) = \gamma_i(S)$ for any $i \geq 2$ where we understand $\gamma_i(G)$ as a subgroup of $S$, cf. also \cite[Theorem 1.3]{SS95}. 
So $\gamma_4(G)$ is a subgroup in $S$ not depending on the choice of the group base in $kG$. Hence we can also assume $\gamma_4(G) = 1$ by considering $S/\gamma_4(S)$. If even $\gamma_3(G) = 1$ holds, 
then we are done by Theorem~\ref{th:Invariants}, so we can assume $\gamma_3(G) \neq 1$.

Denote the generators of $G$ as $c$ and $d$ and assume that the order of $c$ is greater than or equal to the order of $d$. 
So by Proposition~\ref{prop_smallgroupring} we can write $S$ as a semidirect product $G \rtimes A$ where 
$A$ is an abelian group and a set of independent generators of $A$ is given by elements of the form $1+(c-1)^i(d-1)^j$ where $i,j \in \mathbb{Z}_{\geq 0}$, $i+j \geq 2$ and not both $i$ and $j$ are divisible 
by $p$. Furthermore the only generators of $A$ of the form $1+(c-1)^i(d-1)^j$ which are potentially not commuting with $G$ are $1+(c-1)(d-1)$, $1+(c-1)^2$ and $1+(d-1)^2$. For $p=2$ this is even only 
$1+(c-1)(d-1)$. So we have the structure of $S$ given quite explicitly. The proposition will now follow from the following claim. Note that all invariants of $H$ appearing in the claim are known invariants 
of the group bases of the modular group algebras of finite $p$-groups, cf. Theorem~\ref{th:Invariants}.

\textbf{Claim:} If $H$ is a normal subgroup of $S$ generated by two elements such that $\gamma_3(H) \neq 1$, $|H| = |G|$, $Z(H) \cong Z(G)$ and $Z(H)/(Z(H) \cap H') \cong Z(G)/(Z(G) \cap G')$, 
then $H$ is isomorphic to $G$.\\
\textit{Proof:} Let $H$ be a subgroup in $S$ with the claimed properties such that $H = \langle ga, hb \rangle$ with $g,h \in G$ and $a,b \in A$. Note that $\gamma_4(S) =1$, so also $\gamma_4(H) =1$. As $G$ 
is $2$-generated note that $G'/\gamma_3(G) = \langle [c,d] \rangle$ \cite[III, Hilfssatz 1.11]{Hup67}. Moreover $A$ centralizes $G'$. 
Using the commutator identities from Lemma~\ref{lemma_commids} we obtain
$$[ga,hb] = [g,hb]^a[a,hb] = [g,b][g,h]^b[a,b][a,h]^b = [g,b][a,h][g,h].$$
As $\gamma_3(H) \neq 1$ this implies $[g,h] \equiv [c,d]^i \mod \gamma_3(G)$ for some $1\leq i\leq p-1$ and also $G' = H'$, where the last equality means equality of sets. We next show that
$[c, d^p] = 1$ if $p$ is odd and $[c, d^2] = [c,d,d]$ if $p=2$. This also follows from the identities in Lemma~\ref{lemma_commids} as we then get $[c, d^p] = [c,d]^p[c,d,d][c,d^2,d] \ldots [c,d^{p-1},d]$. 
From $[c,d^i,d] = [c,d,d]^i$ (using $\gamma_4(G) = 1$) it then follows $[c, d^p] = [c,d,d]^{\binom{p}{2}}$. This in turn implies that $g$ and $h$ are not in $G^p$ and must be generators of $G$. Furthermore 
$(ga)^p = g^pa^p[g,a]^{\binom{p}{2}}$ by Lemma~\ref{lemma_commids}, as $[g,a]$ is central in $S$. Note that if $p = 2$ and $[g,a] \neq 1$ then the order of $g$ is not $2$. Indeed, otherwise the group generated 
by $g$ and $G'$ consists only of involutions as $g^h = g[g,h]$ is also of order $2$. Hence $\circ(ga) = \text{max}\{\circ(a), \circ(g)\}$. As $|G| = |H|$ and $Z(G) \cong Z(H)$ we obtain that 
$\circ(a) \leq \circ(g)$.

Note that $[g,h]$ is not a $p$-th power by the above, as otherwise it would be central. Hence we can write $G$ in terms of generators and relations in the following way where we skip obvious relation such 
as $[g,h] = g^{-1}h^{-1}gh$.
\begin{align*}
G &= \langle g,h,[g,h],[g,h,g],[g,h,h] \ | \\
& g^{\circ(g)}, h^{\circ(h)}, [g,h]^p, [g,h,g]^{\alpha_g}, [g,h,h]^{\alpha_h}, g^{\circ(g)/p}[g,h,g]^{\beta_g}[g,h,h]^{\beta_h}, h^{\circ(h)/p}[g,h,g]^{\gamma_g}[g,h,h]^{\gamma_h} \rangle ,
\end{align*}
where $\alpha_g, \alpha_h \in \{1,p\}$, but not both are $1$, and $\beta_g, \beta_h, \gamma_g, \gamma_h \in \{0,1, \ldots, p-1\}$. From the calculations in the preceding paragraph it follows that 
   $(ga)^{\circ(g)} = (hb)^{\circ(h)} = [ga,hb]^p = [ga,hb,ga]^{\alpha_g} = [ga,hb,hb]^{\alpha_h} = 1$. Assume that $(ga)^{\circ(g)/p}[ga,hb,ga]^{\beta_g}[ga,hb,hb]^{\beta_h} \neq 1$. Then by the above it 
   must equal $a^{\circ(a)/p}$, where $\circ(a) \neq p$. But then $a^{\circ(a)/p} \in H$ is a central element such that we would have $Z(H) \not\cong Z(G)$. Hence the generators of $H$ satisfy exactly the 
   same relations as the generators of $G$ and the groups are indeed isomorphic.	
\end{proof}

We will also use Proposition~\ref{prop_smallgroupring} for groups which are not $2$-generated in Section~\ref{sec:56}.

\section{Quotients of the augmentation ideal and the algorithm of Eick}\label{sec:Algorithm}

A first algorithmic approach to MIP was developed by Roggenkamp and Scott \cite{RS93} and improved and implemented by Wursthorn \cite{Wur90, Wur93}. It is based on the idea to compare $I(kG)/I(kG)^t$ 
with $I(kH)/I(kH)^t$ and if there exists a $t$ such that $I(kG)/I(kG)^t$ and $I(kH)/I(kH)^t$ are not isomorphic, then also $kG$ and $kH$ are not isomorphic. Wursthorn's implementation appears not to be accessible 
anymore nowadays. A new algorithm which iteratively computes a canonical form for $I(kG)/I(kG)^t$ was developed by Eick \cite{Eic08} and implemented in the \texttt{GAP}-package \texttt{ModIsom} \cite{ModIsom}. 

As mentioned in the introduction, our calculations are based on computational improvements to the programs provided in \texttt{ModIsom}. Most crucially the implementation of Eick, applied to a pair of 
groups $G$ and $H$, first computes the presentation of $I(kG)$ and $I(kH)$ and then proceeds to compare the isomorphism types of $I(kG)/I(kG)^t$ and $I(kH)/I(kH)^t$ for increasing $t$. For groups of order $3^7$ 
or $5^6$ the computation of the presentations of $I(kG)$ and $I(kH)$ is very time and memory consuming and we instead calculate only the presentation of $I(kG)/I(kG)^s$ and $I(kH)/I(kH)^s$ for some initial 
value $s$. This value can be increased if needed during the calculations. A natural starting value for $s$ is the Jennings bound of $G$ and $H$. We elaborate on the 
theoretical background which allows an efficient computation of a presentation for $I(kG)/I(kG)^s$. 

Most studies of the MIP are based on the work of Jennings who described an explicit basis for the augmentation ideal of the modular group algebra of a $p$-group \cite{Jen41}. Note that $kG \cong kH$ if 
and only if $I(kG) \cong I(kH)$. For a finite $p$-group $G$ the Jennings-Zassenhaus series is a series of subgroups $G = D_1(G) \geq D_2(G) \ldots \geq D_n(G) = 1$ defined as 
$$D_m(G) = \prod_{ip^j \geq m} \gamma_i(G)^{p^j}$$
where $i \geq 1$ and $j \geq 0$. Note that subsequent quotients in this series are elementary abelian and hence can be viewed as $k$-vector spaces. This series coincides with the series of dimension 
subgroups of $G$ with respect to $kG$, i.e.
$$D_m(G) = G \cap (1+I(kG)^m). $$
If $D_i(G)/D_{i+1}(G)$ has dimension $s$, let $g_{i,j}$ be elements of $G$ such that $\{g_{i,j}D_{i+1}(G) \ | \ 1 \leq j \leq s \}$ forms a basis of $D_i(G)/D_{i+1}(G)$. Note that $s=0$ is possible, 
as $D_i(G)/D_{i+1}(G)$ might be the trivial group. For an element $g \in G$ set $\bar{g} = g- 1 \in kG$ and define the \emph{weight} of $\bar{g}$ as the maximal integer $w$ such that $g \in D_w(G)$. 
Define the weight of an element of shape $\overline{g_1} \ \overline{g_2} \ldots\overline{g_m}$ as the sum of the weights of $\overline{g_j}$ for $1 \leq j \leq m$. Then a basis of $I(kG)^m/I(kG)^{m+1}$ 
is given by elements of type $\prod (\overline{g_{i,j}})^{\alpha_{i,j}}$ of weight $m$ where $i$ and $j$ run over all admissible values and $0 \leq \alpha_{i,j} \leq p-1$. In particular a basis of $I(kG)$ 
is given by all elements of form $\prod (\overline{g_{i,j}})^{\alpha_{i,j}}$ and we call it a \emph{Jennings basis}. The weight of an arbitrary element $x \in I(kG)$ is defined as the minimal weight of an 
element of the Jennings basis which appears in the linear combination of $x$ when expressed in the Jennings basis. See \cite[Chapters 11, 12]{DdSMS99} for more details on dimensions subgroups and the
Jennings-Zassenhaus series.

The Jennings basis allows one to easily compute a presentation for $I(kG)$ in terms of generators and relations. Note that to compute the product of two elements of the Jennings basis the only needed 
knowledge is how elements of the form $\overline{g_{i,j}}$ multiply with each other. For elements $g,h \in G$ we have
\begin{equation}\label{eq:BasicEqInkG}
\overline{gh} = gh-1 = (g-1)(h-1)+(g-1)+(h-1) = \bar{g}\bar{h} + \bar{g} + \bar{h}. 
\end{equation}
Hence if $\bar{g}$ and $\bar{h}$ appear as elements of a Jennings basis such that $\bar{g}\bar{h}$ is not contained in the Jennings basis, i.e. $\bar{g}$ appears later in the ordered basis than $\bar{h}$, then 
\begin{align*}
\bar{g}\bar{h} &= \overline{gh} - \bar{g} - \bar{h} = \overline{hg[g,h]} - \bar{g} - \bar{h} = \bar{h} + \bar{g} + \overline{[g,h]} + \bar{h}\bar{g} + \bar{h}\overline{[g,h]} + \bar{g}\overline{[g,h]} + \bar{h}\bar{g}\overline{[g,h]} - \bar{g} - \bar{h} \\
&= \bar{h}\bar{g} + (1+\bar{h} + \bar{g} + \bar{h}\bar{g})\overline{[g,h]}.
\end{align*}
So to express $\bar{g}\bar{h}$ as a linear combination of Jennings basis elements it remains to express $\overline{[g,h]}$ as such a linear combination. Now if $g \in D_i$ and $h \in D_j$, 
then $[g,h] \in D_{i+j}$ and hence $\overline{[g,h]}$ has bigger weight than $\bar{g}$ and $\bar{h}$. Thus we can continue the process for $\overline{[g,h]}$ until we reach a trivial commutator.

A special situation appears if we want to compute $\bar{g}^p$. Assume $g^p = g_1^{\alpha_1} \ldots g_m^{\alpha_m}$ for certain elements $g_1,\ldots,g_m$ such that $\overline{g_i}$ is an element of 
the Jennings basis and $1 \leq \alpha_i \leq p-1$ for every $1 \leq i \leq m$, then using $\bar{g}^p = \overline{g^p}$ we obtain from \eqref{eq:BasicEqInkG} that
$$\bar{g}^p = \sum_{\substack{M \subseteq \{1,\ldots,m\} \\ M \neq \emptyset}} \prod_{j \in M} \overline{g_j}^{\alpha_j}. $$
So the information needed to effectively multiply in the Jennings basis is given by the expression in the Jennings basis of the commutators $\overline{[g,h]}$ and the $p$-th powers $\bar{g}^p$.

If we are only interested in a presentation of $I(kG)/I(kG)^s$ for a certain integer $s$, then we can drop the summands of $\overline{[g,h]}$ and $\bar{g}^p$ of weight at least $s$, so that the 
multiplication becomes much more efficient. In most cases for two given $p$-groups $G$ and $H$ of the same order the smallest number $r$ satisfying $I(kG)/I(kG)^r \not\cong I(kH)/I(kH)^r$ is much 
smaller than the smallest number $t$ such that $I(kG)^t = 0$. So to solve MIP for a given pair of groups one does not need to start with a multiplication table of $I(kG)$ but can start with a much 
smaller quotient. 

We give a sketch of our algorithm.

\begin{algorithm}[H]%
 \KwData{List $L$ of groups or group id's, integers $s$, $t$, $m$ for starting, step and maximal level}
 \KwResult{An integer $r$ such that $I(kG)/I(kG)^r \not\cong I(kH)/I(kH)^r$ for any $G,H \in L$ or \texttt{false}}
 LT := [ ];\\
 \For{G in L}{Add to LT table containing the basic information and multiplication for $I(kG)/I(kG)^s$}
 $a$ := $s$	\\ 
 \For{$i=1,\ldots,m$}{
   \If{$i > a$}{$a := a+t$ \\ Recompute LT to contain information for $I(kG)/I(kG)^{a}$}
   \For{G in L}{Compute canonical form of $I(kG)/I(kG)^i$ from the information in LT using \texttt{ModIsom}}
   \If{All canonical forms are different}{$r$ := $i$ \\ \texttt{break}} 
 }

 \texttt{return} If defined, $r$ s.t. $I(kG)/I(kG)^r \not\cong I(kH)/I(kH)^r$ for any $G,H \in L$, otherwise \texttt{false}
 \caption{MIPBinSplit}
\end{algorithm}


\section{A question of Bleher, Kimmerle, Roggenkamp \& Wursthorn}\label{sec:JenningsBound}
We recall the question of Bleher, Kimmerle, Roggenkamp \& Wursthorn from Section \ref{sec:ques}. In fact, the original question from \cite{BKRW99} contains two parts and we give a negative answer to the 
first part of the question.

\begin{question*}
 (Bleher-Kimmerle-Roggenkamp-Wursthorn) Let $G$ and $H$ be non-isomorphic $p$-groups [of the same order] such that $D_{n+1}(G) = 1$. Does $I(kG)/I(kG)^{2n+1} \not \cong I(kH)/I(kH)^{2n+1}$ always hold? 
 Does it hold if $G/D_n(G) \cong H/D_n(H)$ and $D_n(G)$ is cyclic?
\end{question*}

The condition that $G$ and $H$ are of the same order is not stated in \cite{BKRW99}, but probably assumed. It is needed 
as can be seen with easy examples such as $G$ and $H$ being cyclic groups of order $p$ and $p^2$ respectively for an odd prime.
In \cite{BKRW99} it is also stated that this bound is not reached by groups of order $2^6$, contradicting our 
calculations\footnote{There are 8 pairs of groups which reach this bound, namely \texttt{ [15, 16], [46, 47], [46, 48], [46, 49], [47, 48], [106, 107],  [161, 162], [164, 165]}.}. However, in regard to 
pairs of groups which share the known group-theoretical invariants this is true. 

We are able to give a negative answer to the first part of Question~\ref{que:BKRW}.

\begin{proposition}\label{prop:AnswerBKRW}
 Let $G = SG(2^8,i)$ and $H=SG(2^8,j)$ where 
 \[(i,j) \in \{(1272, 1273 ), (1274, 1275), (1778, 1781), (1779, 1780)  \}.\]
  Then $D_5(G) = 1$ and $G/D_4(G) \cong H/D_4(H)$ while
\[I(kG)/I(kG)^9 \cong I(kH)/I(kH)^9 \ \ \text{and} \ \ I(kG)/I(kG)^{10} \not\cong I(kH)/I(kH)^{10} .\]
\end{proposition}
For all the groups $G$ appearing in Proposition~\ref{prop:AnswerBKRW} the last non-trivial dimension subgroup $D_4(G)$ is elementary abelian of rank $2$. So these groups do not provide a negative answer 
to the second part of Question~\ref{que:BKRW}. We also did not encounter any other groups in our calculations which would give such an answer. 

In \cite{BKRW99} only $2$-groups are considered, while we work mostly with odd primes. For all groups of order $3^6$, $3^7$ and $5^6$ the bound of Question~\ref{que:BKRW} is not reached. In the notation 
of Question~\ref{que:BKRW} one always has $I(kG)/I(kG)^{2n} \not \cong I(kH)/I(kH)^{2n}$ in all our observed examples. Even a stronger bound holds in our examples, though admittedly we did not study primes 
bigger than $5$ systematically.

\begin{question}\label{que:Ours}
Let $p$ be an odd prime and $G$ and $H$ $p$-groups of the same order. Let $s$ be the Jennings bound of $G$ and $H$ and set $m = \max\left\{s, \ 2s-\frac{p-1}{2} \right\}$. 
Does $I(kG)/I(kG)^{m+1} \not \cong I(kH)/I(kH)^{m+1}$ always hold?
\end{question} 

The bound in Question~\ref{que:Ours} is reached by both $3$-groups and $5$-groups we studied, but not violated by any of them, cf. Section~\ref{sec:SmallOrder} for more details. Proposition~\ref{prop:AnswerBKRW} 
suggests to also include other data than just the prime and the Jennings bound in questions of the type as \ref{que:BKRW} and \ref{que:Ours}. A new systematic study of groups of order $2^9$ might be a start to 
obtain new ideas in this direction. 

The Jennings bound was also useful to detect an error in \texttt{ModIsom}. Namely for $G = SG(3^6,14)$ and $H=SG(3^6,15)$ we have $G/D_4(G) \cong H/D_4(H)$ while \texttt{ModIsom} 
returned $kG/I(kG)^4 \not \cong kH/I(kH)^4$ when using the function \texttt{CheckBin}.

\section{Groups of small order}\label{sec:SmallOrder}

In this section we describe the computational results we achieved with our improvements of Eick's \texttt{ModIsom}. We include an overview of groups of all sizes for which a proof was given with 
the aid of computers. Note that for $2^6$ a proof was given which used the computer only to compute group-theoretical invariants \cite{HS06}. We call a set of groups a \emph{bin}, if the groups 
share all the group-theoretical invariants listed in Theorem~\ref{th:Invariants}.

\begin{table}[h]
\begin{tabular}{c|c|c|c|c}
Size of group & \# All groups & \# Bins &  Maximal size of bin & \# Groups in bins  \\ \hline
$2^6$ & 267 & 8 & 3 & 17 \\
$2^7$ & 2328 & 158 & 5 & 343 \\
$2^8$ & 56092 & 4406 & 14 & 11476 \\ \hline
$3^6$ & 504 & 18 & 2 & 36 \\
$3^7$ & 9310 & 444 & 21 & 1236 \\ \hline
$5^6$ & 684 & 82 & 19 & 343 \\
\end{tabular}\caption{Basic parameters of our calculations}\label{tb:Numbers}
\end{table}

To see how useful the newly included group-theoretical invariants described in Theorem~\ref{th:Invariants} are, one can compare the numbers in Table~\ref{tb:Numbers} with the numbers in \cite[p. 3908]{Eic08}. 
The results of our calculations are also included as data in the package \cite{ModIsomExt}.

\subsection{Groups of orders studied before}

For groups of order $2^6$ we recover the results of Wursthorn \cite{Wur93} which were recalculated in \cite{HS06}. For groups of order $2^7$ we recover the result of \cite{BKRW99}. 
For groups of order $3^6$ and $2^8$ we recover the proof of the MIP for these groups from the first usage of Eick's algorithm \cite{Eic08}. With regard to the Jennings bound we find no counterexamples to Questions~\ref{que:BKRW} or \ref{que:Ours} in the class of groups of order $2^6$, $3^6$ or $2^7$. 
Both bounds are sharp for certain groups of these orders, namely for order $3^6$ the bound is reached by the pair $[97,98]$ while for $2^7$ it is reached by the pairs 
$[555,556]$, $[886,891]$, $[887,893]$, $[1644,1645]$, $[1683,1684]$, $[1815,1816]$ and $[1830,1831]$,  

The groups of order $2^8$ for which the bound of Question~\ref{que:BKRW} is violated are listed in Proposition~\ref{prop:AnswerBKRW}.

There are some differences compared to the observations in \cite{Eic08}. Although we have to deal with way fewer groups of order $3^6$ compared to \cite{Eic08}, we still can not 
confirm the claim that among the groups sharing all group-theoretical invariants the pair $(SG(3^6, 63), SG(3^6,64))$ is the only one not satisfying $kG/I(kG)^{10} \not\cong kH/I(kH)^{10}$. 
E.g. for $G = SG(3^6, 97)$ and $H = SG(3^6, 98)$ we have $I(kG)/I(kG)^{17} \cong I(kH)/I(kH)^{17}$ (the Jennings bound of $G$ and $H$ is 9).

Unfortunately we lack the computational possibilities to reevaluate the MIP for groups of order $2^9$ which has been studied in \cite{EK11}. This remains to be done.

\subsection{Groups of order $3^7$}
For groups of order $3^7$  we obtain a full result showing that the MIP has a positive answer for all groups of this order. Bagi\'nski's observation proved in Section~\ref{sec:Bag} turns 
out to be very useful as it allows to avoid the application of the algorithm to groups with Jennings bound 81.

There is no pair of groups violating the bound of Question~\ref{que:Ours} while there are four bins of groups in which each pair of groups reaches it, namely \texttt{[243, 244]}, 
\texttt{[294, 295, 296, 297, 298, 299]}, \texttt{[358, 359]} and \texttt{[5846, 5847]}.

\subsection{Groups of order $5^6$}\label{sec:56}
Bagi\'nski's observation proved in Section~\ref{sec:Bag} turns out to be very useful as it allows to avoid the application of the algorithm to groups with Jennings bound 125. Moreover, 
for some groups which are not $2$-generated the description of the unit group of the small group ring given in Section~\ref{sec:Bag} still turns out to be useful. Namely 
if $\gamma_2(G)^p\gamma_4(G) = 1$ and $kG \cong kH$, then the unit group $S$ of the small group ring of $kG$ must contain a normal subgroup isomorphic to $H$. Having the explicit description 
of $S$ from Proposition~\ref{prop_smallgroupring} one can use \texttt{GAP} to construct $S$ and look for a copy of $H$ in $S$. The bins which took a long time to finish in our main algorithm, 
but could be handled by this application of Proposition~\ref{prop_smallgroupring} are \texttt{[48, 49], [162, 163], [168, 169]} and \texttt{[550, 551]}. A pair of groups $G$ and $H$ which 
satisfy $\gamma_2(G)^p\gamma_4(G) = 1$ and $\gamma_2(H)^p\gamma_4(H) = 1$, but which have 
isomorphic unit groups of small group rings are \texttt{[553, 554]}\footnote{The isomorphism was detected using \texttt{RandomIsomorphismTest}.}. In fact our calculations were not performed 
for the whole group $S$ but for $S/Z(S)\cap A$ which also contains any group base.

No groups for which we could answer MIP violate the bound expressed in Question~\ref{que:Ours} while there are four bins for which each pair of groups reaches this bound, namely
\texttt{[577,578,579]}, \texttt{[581, 582, 583]}, \texttt{[675, 676, 677, 678, 679]} and \texttt{[681, 682, 683, 684]}.

\begin{remark} An argument which is sometimes used to distinguish $kG$ and $kH$ is to count the number of elements in $kG$ which satisfy a certain condition which is defined independently 
of $G$. This has been done in \cite{Pas65} or for certain groups in \cite{Wur93}, cf. also \cite[Section 2.7]{HS06}. Two maps are typically used:
$$\varphi_{n,m,\ell}: I(kG)^n/I(kG)^{n+m} \rightarrow I(kG)^{np^\ell}/I(kG)^{np^\ell+m}, \ x \mapsto x^{p^\ell} $$
and 
$$\psi_n: Z(kG) \cap I(kG)^n \rightarrow Z(kG), \ x \mapsto x^p $$
and the number of elements mapping to $0$ under these maps are compared. For the open cases of groups of order $5^6$ we did not find integers $n,m,\ell$ which give 
different numbers of elements under $\varphi_{n,m,\ell}$ and moreover all non-trivial class sums in all the groups become $0$ when taken to the $5$-th power, i.e. $\psi_n$ can not be used.
\end{remark}

\textbf{Acknowledgments:} We are very grateful to Bettina Eick for her idea for this collaboration and many useful conversations.

\bibliographystyle{amsalpha}
\bibliography{MIP}

\providecommand{\bysame}{\leavevmode\hbox to3em{\hrulefill}\thinspace}
\providecommand{\MR}{\relax\ifhmode\unskip\space\fi MR }
\providecommand{\MRhref}[2]{%
  \href{http://www.ams.org/mathscinet-getitem?mr=#1}{#2}
}
\providecommand{\href}[2]{#2}
\begin{thebibliography}{DdSMS99}

\bibitem[Bag92]{Bag92}
C.~Bagi\'{n}ski, \emph{Modular group algebras of {$2$}-groups of maximal
  class}, Comm. Algebra \textbf{20} (1992), no.~5, 1229--1241.

\bibitem[Bag99]{Bag99}
\bysame, \emph{On the isomorphism problem for modular group algebras of
  elementary abelian-by-cyclic {$p$}-groups}, Colloq. Math. \textbf{82} (1999),
  no.~1, 125--136.

\bibitem[BC88]{BC88}
C.~Bagi\'{n}ski and A.~Caranti, \emph{The modular group algebras of
  {$p$}-groups of maximal class}, Canad. J. Math. \textbf{40} (1988), no.~6,
  1422--1435.

\bibitem[BdR20]{BdR20}
O.~Broche and \'{A}. del R\'{\i}o, \emph{The modular isomorphism problem for
  two generated groups of class two}, arxiv.org/abs/arXiv:2003.13281 (2020),
  1--9.

\bibitem[BK07]{BK07}
C.~Bagi\'{n}ski and A.~Konovalov, \emph{The modular isomorphism problem for
  finite {$p$}-groups with a cyclic subgroup of index {$p^2$}}, Groups {S}t.
  {A}ndrews 2005. {V}ol. 1, London Math. Soc. Lecture Note Ser., vol. 339,
  Cambridge Univ. Press, Cambridge, 2007, pp.~186--193.

\bibitem[BK19]{BK19}
C.~Bagi\'{n}ski and J.~Kurdics, \emph{The modular group algebras of
  {$p$}-groups of maximal class {II}}, Comm. Algebra \textbf{47} (2019), no.~2,
  761--771.

\bibitem[BKRW99]{BKRW99}
F.~M. Bleher, W.~Kimmerle, K.~W. Roggenkamp, and M.~Wursthorn,
  \emph{Computational aspects of the isomorphism problem}, Algorithmic algebra
  and number theory ({H}eidelberg, 1997), Springer, Berlin, 1999, pp.~313--329.

\bibitem[Bra63]{Bra63}
R.~Brauer, \emph{Representations of finite groups}, Lectures on {M}odern
  {M}athematics, {V}ol. {I}, Wiley, New York, 1963, pp.~133--175.

\bibitem[DdSMS99]{DdSMS99}
J.~D. Dixon, M.~P.~F. du~Sautoy, A.~Mann, and D.~Segal, \emph{Analytic
  pro-{$p$} groups}, second ed., Cambridge Studies in Advanced Mathematics,
  vol.~61, Cambridge University Press, Cambridge, 1999.

\bibitem[Dre89]{Dre89}
V.~Drensky, \emph{The isomorphism problem for modular group algebras of groups
  with large centres}, Representation theory, group rings, and coding theory,
  Contemp. Math., vol.~93, Amer. Math. Soc., Providence, RI, 1989,
  pp.~145--153.

\bibitem[Eic08]{Eic08}
B.~Eick, \emph{Computing automorphism groups and testing isomorphisms for
  modular group algebras}, J. Algebra \textbf{320} (2008), no.~11, 3895--3910.

\bibitem[Eic20]{ModIsom}
\bysame, \emph{Mod{I}som: a {G}{A}{P} 4 package, version 2.5.1},
  https://gap-packages.github.io/modisom/, 2020.

\bibitem[EK11]{EK11}
B.~Eick and A.~Konovalov, \emph{The modular isomorphism problem for the groups
  of order 512}, Groups {S}t {A}ndrews 2009 in {B}ath. {V}olume 2, London Math.
  Soc. Lecture Note Ser., vol. 388, Cambridge Univ. Press, Cambridge, 2011,
  pp.~375--383.

\bibitem[GAP19]{GAP}
The GAP~Group, \emph{{GAP -- Groups, Algorithms, and Programming, Version
  4.10.2}}, 2019, http://www.gap-system.org.

\bibitem[Her01]{Her01}
M.~Hertweck, \emph{A counterexample to the isomorphism problem for integral
  group rings}, Ann. of Math. (2) \textbf{154} (2001), no.~1, 115--138.

\bibitem[Her07]{Her07}
\bysame, \emph{A note on the modular group algebras of odd {$p$}-groups of
  {$M$}-length three}, Publ. Math. Debrecen \textbf{71} (2007), no.~1-2,
  83--93.

\bibitem[HS06]{HS06}
M.~Hertweck and M.~Soriano, \emph{On the modular isomorphism problem: groups of
  order {$2^6$}}, Groups, rings and algebras, Contemp. Math., vol. 420, Amer.
  Math. Soc., Providence, RI, 2006, pp.~177--213.

\bibitem[Hup67]{Hup67}
B.~Huppert, \emph{Endliche {G}ruppen. {I}}, Die Grundlehren der Mathematischen
  Wissenschaften, Band 134, Springer-Verlag, Berlin-New York, 1967.

\bibitem[Jen41]{Jen41}
S.~A. Jennings, \emph{The structure of the group ring of a {$p$}-group over a
  modular field}, Trans. Amer. Math. Soc. \textbf{50} (1941), 175--185.

\bibitem[MM20]{ModIsomExt}
L.~Margolis and T.~Moede, \emph{Mod{I}som{E}xt: a {GAP} 4 package, version
  1.0.0}, https://www.tu-braunschweig.de/en/iaa/personal/moede, 2020.

\bibitem[NS18]{NS18}
G.~Navarro and B.~Sambale, \emph{On the blockwise modular isomorphism problem},
  Manuscripta Math. \textbf{157} (2018), no.~1-2, 263--278.

\bibitem[Pas65]{Pas65}
D.~S. Passman, \emph{The group algebras of groups of order {$p^{4}$} over a
  modular field}, Michigan Math. J. \textbf{12} (1965), 405--415.

\bibitem[RS93]{RS93}
K.~W. Roggenkamp and L.~L. Scott, \emph{Automorphisms and nonabelian
  cohomology: an algorithm}, vol. 192, 1993, Computational linear algebra in
  algebraic and related problems (Essen, 1992), pp.~355--382.

\bibitem[Sak20]{Sak20}
T.~Sakurai, \emph{The isomorphism problem for group algebras: a criterion}, J.
  Group Theory \textbf{23} (2020), no.~3, 435--445.

\bibitem[San85]{San84}
R.~Sandling, \emph{The isomorphism problem for group rings: a survey}, Orders
  and their applications ({O}berwolfach, 1984), Lecture Notes in Math., vol.
  1142, Springer, Berlin, 1985, pp.~256--288.

\bibitem[San89]{San89}
\bysame, \emph{The modular group algebra of a central-elementary-by-abelian
  {$p$}-group}, Arch. Math. (Basel) \textbf{52} (1989), no.~1, 22--27.

\bibitem[San96]{San96}
\bysame, \emph{The modular group algebra problem for metacyclic {$p$}-groups},
  Proc. Amer. Math. Soc. \textbf{124} (1996), no.~5, 1347--1350.

\bibitem[SS95]{SS95}
M.~A.~M. Salim and R.~Sandling, \emph{The unit group of the modular small group
  algebra}, Math. J. Okayama Univ. \textbf{37} (1995), 15--25.

\bibitem[SS96]{SS96}
\bysame, \emph{The modular group algebra problem for groups of order {$p^5$}},
  J. Austral. Math. Soc. Ser. A \textbf{61} (1996), no.~2, 229--237.

\bibitem[Wur90]{Wur90}
M.~Wursthorn, \emph{Die modularen {G}ruppenringe der {G}ruppen der {O}rdnung
  $2^6$}, Master's thesis, Universit\"at Stuttgart, 1990.

\bibitem[Wur93]{Wur93}
\bysame, \emph{Isomorphisms of modular group algebras: an algorithm and its
  application to groups of order {$2^6$}}, J. Symbolic Comput. \textbf{15}
  (1993), no.~2, 211--227.

\end{thebibliography}

\end{document}